\documentclass[12pt,a4paper]{amsart} 
\usepackage[utf8]{inputenc}
\usepackage[english]{babel}

\usepackage{amsmath}
\usepackage{amsfonts}
\usepackage{amssymb}
\usepackage{color}

\usepackage{amsthm}

\newtheorem{theorem}{Theorem} [section]
\newtheorem{proposition}[theorem]{Proposition}
\newtheorem{corollary}[theorem]{Corollary} 
\newtheorem{lemma}[theorem]{Lemma}
\newtheorem{conjecture}[theorem]{Conjecture}

\newcounter{defno}
\newcommand{\definition}[1]{\addtocounter{defno}{1}\textbf{Definition \thedefno.}#1}

\usepackage{graphics}
\usepackage{epsfig}

\usepackage{url}
\usepackage{hyperref}

\usepackage[a4paper,top=2.5cm,bottom=2.8cm,
left=3.2cm,right=3.2cm]{geometry}
\markleft{\hfill \textsc{Arrangement of Central Points on the Faces of a Tetrahedron} \hfill}
\long\def\void#1{}

\begin{document}
International Journal of  Computer Discovered Mathematics (IJCDM) \\
ISSN 2367-7775 \copyright IJCDM \\
Volume 5, 2020, pp. 13--41  \\
Received 6 August 2020. Published on-line 30 September 2020 \\ 
web: \url{http://www.journal-1.eu/} \\
\copyright The Author(s) This article is published 
with open access\footnote{This article is distributed under the terms of the Creative Commons Attribution License which permits any use, distribution, and reproduction in any medium, provided the original author(s) and the source are credited.}. \\
\bigskip
\bigskip

\begin{center}
	{\Large \textbf{Arrangement of Central Points\\on the Faces of a Tetrahedron}} \\
	\medskip
	\bigskip

	\textsc{Stanley Rabinowitz} \\

	545 Elm St Unit 1,  Milford, New Hampshire 03055, USA \\
	e-mail: \href{mailto:stan.rabinowitz@comcast.net}{stan.rabinowitz@comcast.net} \\
	web: \url{http://www.StanleyRabinowitz.com/} \\

\end{center}
\bigskip

\textbf{Abstract.} We systematically investigate properties of various triangle centers
(such as orthocenter or incenter)
located on the four faces of a tetrahedron.
For each of six types of tetrahedra, we examine over 100 centers
located on the four faces of the tetrahedron.
Using a computer, we determine when any of 16 conditions occur (such as the four
centers being coplanar).
A typical result is:
The lines from each vertex of a circumscriptible tetrahedron
to the Gergonne points of the opposite face are concurrent.




\def\T{^{\rm T}}


\bigskip
\bigskip
%
\section{Introduction}
\label{section:introduction}

Over the centuries, many notable points have been found
that are associated with an arbitrary triangle. Familiar examples include:
the centroid, the circumcenter, the incenter, and the orthocenter.
Of particular interest are those points that
Clark Kimberling classifies as ``triangle centers''.
He notes over 100 such points in his seminal paper \cite{KimberlingA}.

Given an arbitrary tetrahedron and a choice of triangle center
(for example, the circumcenter), we may locate this triangle center
in each face of the tetrahedron.  We wind up with four points,
one on each face.  What can be said about these points? For example,
do the 4 points form a tetrahedron similar to the original one?
Could these 4 points ever lie in a plane?  Might they form
a regular tetrahedron? Consider the 4 lines from the
vertices of the tetrahedron to the centers in the opposite faces.
Do these 4 lines concur? Might they have the same length?

In this paper, we investigate such questions for a large
collection of triangle centers
and for various types of tetrahedra.

A typical result is:
The lines from each vertex of a circumscriptible tetrahedron
to the Gergonne points of the opposite face are concurrent.

For information about what you need to know about triangle centers and center functions,
we give a short overview in Appendix \ref{section:trilinearCoordinates}.

We make extensive use of areal coordinates (also known as barycentric coordinates) when
analyzing points associated with triangles, such as the faces of a tetrahedron.
For the reader not familiar with areal coordinates, we give the basics in Appendix \ref{section:arealCoordinates}.

For points, lines, and planes in space, we make heavy use of tetrahedral coordinates.
For the reader not familiar with tetrahedral coordinates, we present the needed information
in Appendix \ref{section:tetrahedralCoordinates}.

Throughout this paper, the notation $[XYZ]$ will denote the area of triangle XYZ.

\section{Coordinates for the Face Centers}
\label{section:faceCenters}

When referring to an arbitrary tetrahedron (the reference tetrahedron), we will usually label
the vertices $A_1$, $A_2$, $A_3$, and $A_4$. The lengths of the
sides of the base ($\triangle A_1A_2A_3$) will be $a_1$, $a_2$, and $a_3$,
with edge $a_i$ opposite vertex $A_i$.  In the tetrahedron, the edge
opposite the edge of length $a_i$ will have length $b_i$. See Figure \ref{fig:isoscelesTetrahedron}a.

\begin{figure}[h!t]
\centering
\includegraphics[width=0.67\linewidth]{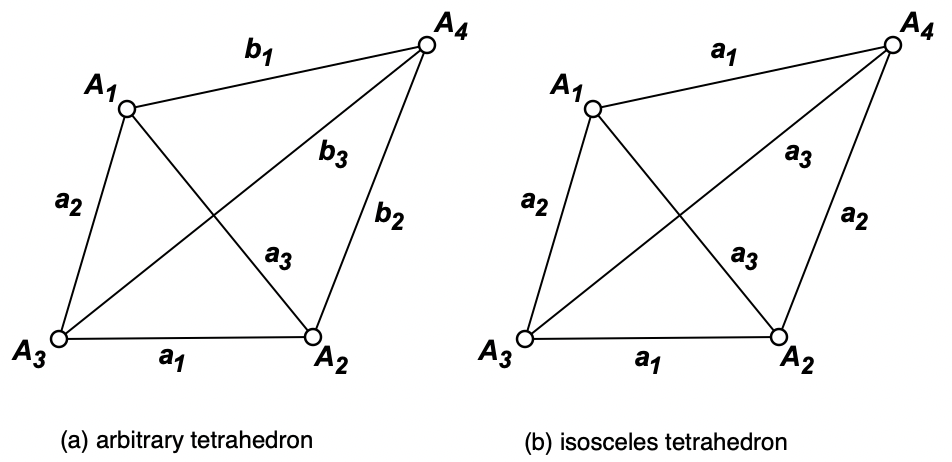}
\caption{edge labeling}
\label{fig:isoscelesTetrahedron}
\end{figure}

Thus, we have
$$A_2A_3=a_1,\quad A_3A_1=a_2,\quad A_1A_2=a_3,\quad
A_1A_4=b_1,\quad A_2A_4=b_2,\quad A_3A_4=b_3.$$

If the tetrahedron has its opposite edges of equal length, then
the tetrahedron is called an isosceles tetrahedron. See Figure \ref{fig:isoscelesTetrahedron}b
(not to scale). It is clear that in an isosceles tetrahedron,
the four faces are congruent because they each have sides
of length $a_1$, $a_2$, $a_3$. In a sense, the isosceles
tetrahedon ``looks the same'' from each vertex. This four-fold
symmetry makes the isosceles tetrahedron be the figure in space
that corresponds to the equilateral triangle in the plane.
An equilateral triangle has 3 identical sides and an isosceles
tetrahedron has 4 identical faces.
The face of the tetrahedron opposite vertex $A_i$ will be called face $i$
of the tetrahedron. Its area will be denoted by $F_i$.

If $(x_1,x_2,x_3)$ are the areal coordinates for a triangle center, then
the tetrahedral coordinates for the corresponding center on face 4
of our reference tetrahedron $A_1A_2A_3A_4$ is $(x_1,x_2,x_3,0)$.
To see why this is true, consider the center
$P$ on face 4 (triangle $A_1A_2A_3$) of the tetrahedron. Then
$$[PA_1A_2A_4]:[PA_2A_3A_4]:[PA_3A_1A_4]=[PA_1A_2]:[PA_2A_3]:[PA_3A_1]$$
since these 4 tetrahedra have a common altitude from $A_4$.

We will frequently
have occasion to pick a point on each face of the tetrahedron. In such
a case, the point on face $i$ will be labelled $P_i$.
It is often necessary to locate such a point based on its areal
coordinates in face $i$.  We must be careful how we set up the
coordinate system on each face.
Note that in an arbitrary tetrahedron, each face has the property
that the labels associated with each edge ($a_i$ or $b_i$) contains
one label with subscript 1, one with subscript 2, and one with
subscript 3.
In order to maintain the 4-fold
symmetry exhibited by an isosceles tetrahedron, under the mapping
$A_1\to A_2\to A_3\to A_4$
we want the faces to transform as follows:
$$\triangle A_4A_3A_2\to\triangle A_3A_4A_1\to
\triangle A_2A_1A_4\to\triangle A_1A_2A_3.$$
Note that in each face, our labelling starts with the vertex opposite
the edge whose label has subscript 1, then proceeds to the vertex
opposite the edge whose label has subscript 2 and finally ends with
the edge whose label has subscript 3. This induces the following
correspondence between the edges:
$$
\begin{aligned}
(a_1,b_2,b_3,b_1,a_2,a_3)&\to
(b_1,a_2,b_3,a_1,b_2,a_3)\\
&\to
(b_1,b_2,a_3,a_1,a_2,b_3)\\
&\to
(a_1,a_2,a_3,b_1,b_2,b_3).
\end{aligned}
$$

This mapping is shown in figure \ref{fig:mapping}.
\begin{figure}[h!t]
\centering
\includegraphics[width=1\linewidth]{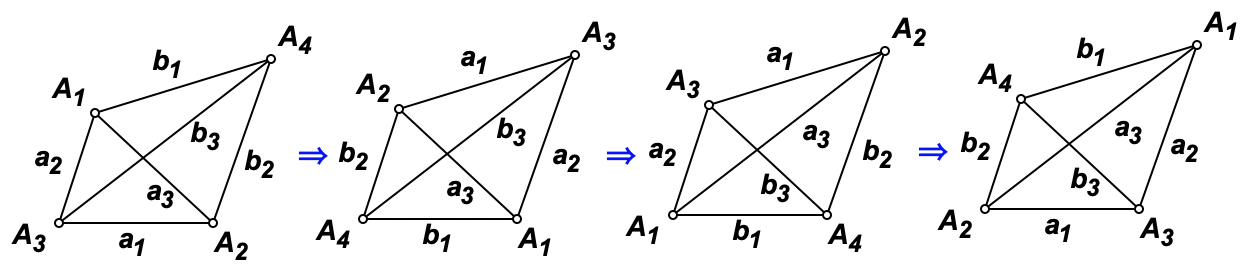}
\caption{mapping}
\label{fig:mapping}
\end{figure}

Let us now consider the mapping which takes $A_1$ into $A_2$
in this 4-fold symmetry.  We start, by finding the tetrahedral coordinates
for the point, $P_1$,  with areal coordinates
$(x_1,x_2,x_3)$ in face 1 of the reference tetrahedron.
The first coordinate ``$x_1$'' refers to the area formed by the point $P_1$
and the side of the triangle with a ``1'' as subscript. In this case,
face 1 has sides of length $a_1$, $b_2$, and $b_3$, so the side we need
is the side of length $a_1$. On face 1, this side is opposite vertex $A_4$
of the reference tetrahedron and so the ``$x_1$'' coordinate will appear as the
4th tetrahedral coordinate. Proceeding in this manner, we find that $P_1$
has tetrahedral coordinates $(0,x_3,x_2,x_1)$.
This point wants to map to a point, $P_2$, with the same areal
coordinates in face 2. In face 1 ($\triangle A_4A_3A_2$), the
coordinates correspond to areas associated with edges $a_1$, $b_2$,
and $b_3$. In face 2 ($\triangle A_3A_4A_1$), the corresponding
edges are $b_1$, $a_2$, and $b_3$. Point $P_2$ has tetrahedral coordinates
$(x_3,0,x_1,x_2)$ because on face 2, edge $b_1$ is opposite vertex $A_3$
(so $x_1$ moves to the 3rd coordinate in the tetrahedral system),
$a_2$ is opposite vertex $A_4$
(so $x_2$ moves to the 4th coordinate in the tetrahedral system),
and $b_3$ is opposite vertex $A_1$
(so $x_3$ moves to the 1st coordinate in the tetrahedral system).

In other words, given a point in the plane with areal coordinates
$(x_1,x_2,x_3)$, the corresponding points in the faces
of our reference tetrahedron are:\par
\vbox{
$${\rm Face\ 1:}\qquad (0,x_3,x_2,x_1)$$
$${\rm Face\ 2:}\qquad (x_3,0,x_1,x_2)$$
$${\rm Face\ 3:}\qquad (x_2,x_1,0,x_3)$$
$${\rm Face\ 4:}\qquad (x_1,x_2,x_3,0)$$
where we have associated face 4 with the original plane triangle.
}

If the original point is a center, with areal coordinates
$$\left(f(a_1,a_2,a_3),f(a_2,a_3,a_1),f(a_3,a_1,a_2)\right),$$
then the corresponding points on the faces of the tetrahedron are:
$$
\begin{aligned}
{\rm Face\ 1:}\qquad &(0,f(b_3,b_2,a_1),f(b_2,a_1,b_3),f(a_1,b_3,b_2))\\
{\rm Face\ 2:}\qquad &(f(b_3,b_1,a_2),0,f(b_1,a_2,b_3),f(a_2,b_3,b_1))\\
{\rm Face\ 3:}\qquad &(f(b_2,b_1,a_3),f(b_1,a_3,b_2),0,f(a_3,b_2,b_1))\\
{\rm Face\ 4:}\qquad &(f(a_1,a_2,a_3),f(a_2,a_3,a_1),f(a_3,a_1,a_2),0).
\end{aligned}
$$

Kimberling \cite{KimberlingB} and \cite{KimberlingC} has collected the trilinear coordinates for over 40,000
centers associated with a triangle.  He lists the trilinear coordinates
in terms of the sides $a$, $b$, $c$ of the reference triangle and
trigonometric functions of $A$, $B$, $C$, the angles of the reference
triangle.  Only the first coordinate is given, for if this coordinate
is $f(a,b,c,A,B,C)$, then the other coordinates are $f(b,c,a,B,C,A)$
and $f(c,a,b,C,A,B)$ respectively.

We wish to study points associated with a tetrahedron based on the
lengths of the 6 edges of the tetrahedron. The six edge lengths are
independent quantities. Involving other quantities such as the face
areas or trigonometric functions of the face or dihedral angles
would yield expressions containing dependent variables and would
complicate the process of determining if such expressions are identically
0 for all tetrahedra.  We thus need to remove the presence of angles
from Kimberling's data.  Since all the trigonometric functions present
can be expressed in terms of sine's and cosine's of the angles
of the reference triangle, the following formulas suffice to
remove all reference to these angles:
$$
\begin{aligned}
\sin A&=\frac{2K}{bc}\\
\cos A&=\frac{b^2+c^2-a^2}{2bc}\
\end{aligned}
$$
where $K$ denotes the area of the reference triangle.
The first formula comes from the well-known area formula: $K=\frac12 bc\sin A$;
and the second formula is The Law of Cosines.
Similar expressions hold for angles $B$ and $C$.

Factors (such as $K$ or $a+b+c$) that would be common to all three
coordinates are then removed.

The presence of a $K$ in the denominator of
any fraction involved is cumbersome and was removed by replacing terms of the form
$x/(y+z K)$ by $x(y-z K)/(y^2-z^2K^2)$.
This leaves all square roots in the numerators.

The variable $K$ is then replaced by its equivalent expression
in terms of the sides of the triangle (Heron's Formula), namely
$$K=\frac14\sqrt{2a^2b^2+2b^2c^2+2c^2a^2-a^4-b^4-c^4}.$$


Since tetrahedral coordinates are barycentric, if $(x,y,z,w)$
are the coordinates for one of the above centers in our reference
tetrahedron, then the tetrahedral coordinates for the corresponding
center in a tetrahedron with vertices $P_1$, $P_2$, $P_3$, and $P_4$
are $xP_1+yP_2+zP_3+wP_4$, where $xP_1$ denotes the scalar product
of $x$ and the vector $P_1$, etc. Algorithmically, the desired
coordinates are the dot product of the vectors $(x,y,z,w)$ and
$(P_1,P_2,P_3,P_4)$.

\section{Types of Tetrahedra}
\label{section:typesOfTetahedra}

The types of tetrahedra investigated are listed in the following table.
\smallskip

\begin{center}
\footnotesize
\begin{tabular}{|l|l|l|}\hline
\multicolumn{3}{|c|}{\small\textbf{Types of Tetrahedra Considered}}\\ \hline
Tetrahedron Type&Geometric Definition&Algebraic Condition $i=1,2,3$\\ \hline
General&no restrictions placed on the edges&none\\ \hline
Isosceles&faces are congruent&$a_i=b_i$\\ \hline
Circumscriptible&edges are tangent to a sphere&$a_i+b_i=$ constant\\ \hline
Isodynamic&symmedians are concurrent&$a_ib_i=$ constant\\ \hline
Orthocentric&opposite edges are perpendicular&$a_i^2+b_i^2=$ constant\\ \hline
Harmonic&n/a&$1/a_i+1/b_i=$ constant\\ \hline
\end{tabular}
\end{center}

\smallskip
Only tetrahedra that have the requisite 4-fold symmetry were studied.
Thus, for example, trirectangular tetrahedra are not included in this study.
We would have liked to have investigated isogonic tetrahedra (ones
in which the cevians to the points of tangency of the insphere are concurrent,
(\cite[p.~328]{Altshiller}),
but the corresponding algebraic condition was too messy to be manageable.
The concept of a harmonic tetrahedron was invented for this study
and has a few interesting properties, but perhaps not enough to warrant
future study. The other types of tetrahedra are well known and
information about them can be found in \cite{Altshiller}.

\section{Methodology}
\label{section:methodology}

For this study, we considered the first 101 triangle centers listed in \cite{KimberlingB},
X1 through X101, as well as a few other centers, listed in the following table.

\begin{center}
\begin{tabular}{|l|l|}\hline
\multicolumn{2}{|c|}{\textbf{Triangle Centers Considered}}\\ \hline
Triangle Center&Trilinears\\
X1--X101&see \cite{KimberlingB}\\
Y1&$1/(a^2(b+c)-a b c)$\\
Y2&$a/((b-a)(c-a))$\\
Y3&$a^2(b+c)$\\
Y4&$1/(a^2(b+c))$\\
Y5&$a/(b^2+c^2)$\\
Y6&$a^2(b^2+c^2)$\\
Y7&$1/(a^2(b^2+c^2))$\\
Y8&$(b^2+c^2)/a$\\
Y9&$a(b+c-2a)$\\
Y10&$1/(a(b+c-2a))$\\
Y11&$a/(b+c-2a)$\\
Y12&$a^2(b+c-2a)$\\
Y13&$1/(a^2(b+c-2a))$\\
Y14&$b+c-b c/a$\\ \hline
r-power point&$a^r$\\
Z1&$a^r(b+c)$\\
Z2&$a^r(b^2+c^2)$\\
Z3&$a^r(b+c-a)$\\
Z4&$a^r(b+c-2a)$\\
Z5&$a^r(b^2+c^2-a^2)$\\
Z6&$a^r(b^3+c^3)$\\
Z7&$a^r(b^2+c^2+b c)$\\
Z8&$2a^r+b^r+c^r$\\
Z9&$(b^r+c^r)/a$\\
Z10&$(b^r+c^r-a^r)/a$\\
Z11&$(b^r+c^r+2a^r)/a$\\ \hline
power center&$a^r g[b,c]$\\
arbitrary center&$f[a,g[b,c]]$\\
areal center&$f[a,b,c]/a$\\ \hline
\end{tabular}
\end{center}

For each type of tetrahedron considered,
and for each triangle center considered,
we computed the tetrahedral coordinates of these centers on each face of
the tetrahedron.

Once we had located these four centers, we then used Mathematica to
run a barrage of tests on these four points to see if they satisfied any
special properties. Since these tests involved algebraic coordinates
(i.e. we were not looking at specific tetrahedra with numerical sides),
any results found constitute a proof that the result is true and
not merely a conjecture based on numerical evidence. These results are stated
in sections \ref{section:tetrahedra} through \ref{section:harmonic}. The proofs are by coordinate geometry, mechanically
performed by the Mathematica program which was written to compute
all the necessary lengths and coordinates and then confirm the claimed
results symbolically.

Most of these results are new, however, some of them may have previously appeared
in the literature. We give references, when known. A few related results appeared as problem 3 in the 15th Summer Conference of the International Mathematical Tournament of Towns, \cite{Towns}.

\medskip
First a few definitions.

\definition{
The original tetrahedron is known as the {\it reference tetrahedron}.
}

\definition{
The tetrahedron formed by the four centers is called the {\it central tetrahedron}.
}

\definition{
The line segment from a vertex of the reference tetrahedron to the
center on the opposite face is called a {\it cevian}.
}

\definition{
Four skew lines in space
are said to form a {\it hyperbolic group} if there is an infinite number of lines
that meet all four of these lines.
}

According to Altshiller-Court (\cite[p.~10]{Altshiller}),
``Such a group is often the space analog of three concurrent lines in the plane.''

\definition{
The four skew lines are part of an infinite family of lines that form a
ruled surface known as a {\it hyperboloid} of one sheet.
}

\definition{
By a {\it space center} of a tetrahedron, we mean
one of: centroid, circumcenter, incenter, Monge point, or Euler point.
These are described in the following table.
}

\begin{center}
\begin{tabular}{|l|l|}\hline
\multicolumn{2}{|c|}{\textbf{Tetrahedron Centers Considered}}\\ \hline
Space Center&Description\\ \hline
centroid&intersection point of medians\\
circumcenter&center of circumscribed sphere\\
incenter&center of inscribed sphere\\
Monge point&symmetric of circumcenter with respect to centroid\\
Euler point&center of 12-point sphere\\ \hline
\end{tabular}
\end{center}

More background information about these centers is given in Appendix \ref{section:tetrahedronCenters}.

\smallskip
The properties that were checked for are listed in the table below.
\smallskip

\begin{center}
\begin{tabular}{|l|l|}\hline
\multicolumn{2}{|c|}{\textbf{Properties Considered}}\\ \hline
Property \hphantom{0}1&The cevians to the four centers are concurrent.\\
Property \hphantom{0}2&The cevians to the four centers form a hyperbolic group.\\
Property \hphantom{0}3&The four centers are coplanar.\\
Property \hphantom{0}4&The four centers are collinear.\\
Property \hphantom{0}5&The normals to the faces at the centers concur.\\
Property \hphantom{0}6&The faces of the central tetrahedron are parallel to the faces\\
&\qquad of the reference tetrahedron.\\
Property \hphantom{0}7&The central tetrahedron is isosceles.\\
Property \hphantom{0}8&The central tetrahedron is regular.\\
Property \hphantom{0}9&The central tetrahedron is isodynamic.\\
Property 10&The central tetrahedron is circumscriptible.\\
Property 11&The central tetrahedron is orthocentric.\\
Property 12&The central tetrahedron is similar to the reference tetrahedron.\\
Property 13&The cevians to the four centers have the same length.\\
Property 14&The central tetrahedron has a space center in common\\
&\qquad with some space center of the reference tetrahedron.\\
Property 15&The central tetrahedron has a space center on the Euler line\\
&\qquad of the reference tetrahedron.\\
Property 16&The reference tetrahedron has a space center on the Euler line\\
&\qquad of the central tetrahedron.\\ \hline
\end{tabular}
\end{center}

\smallskip
If the cevians concurred, we also checked to see if the point of
concurrence was a space center of the reference tetrahedron or if
it lied on the Euler line of the reference tetrahedron.
Also, if the four cevians formed a hyperbolic group, we computed the
center of the hyperboloid for which these cevians were generators
and checked this point to see if it was a space center of the
reference tetrahedron (or on its Euler line).

To find the center of the hyperboloid, we used the following result:

\begin{proposition}[\cite{Lee}] If $L_1$, $L_2$, and $L_3$ are three lines
that determine a hyperboloid of one sheet, then if one draws planes
through each of these lines parallel to the two others, then we
get a parallelepiped. The center of this parallelepiped is the
center of the hyperboloid.
\end{proposition}

Thus, our test was as follows: Use formula 17 to find the plane through
$L_1$ and parallel to $L_2$. (See Appendix \ref{section:formulas} for formulas using
tetrahedral coordinates.) Let $X$ be the point of
intersection of $L_3$ with this plane (found via formula 18).
Similarly, find the plane through $L_1$ parallel to $L_3$. Let $Y$ be
the intersection of this plane and $L_2$.
Then the center of the hyperboloid is the midpoint of segment $XY$.

To make some of the computation of properties 1-16 easier, we first checked the property
for a specific tetrahedron with numerical sides. If the property was false
(using exact arithmetic) for this numerical case, then we did not bother
checking to see if the property was algebraically true in general.

\section{Results found for Arbitrary Tetrahedra}
\label{section:tetrahedra}

The following results were discovered and proven by our computer program.

\begin{theorem}
Consider the centroids on each face of an arbitrary tetrahedron.
Then\\
(a) The faces of the central terahedron are parallel to the corresponding
faces of the reference tetrahedron.\\
(b) The cevians to the centroids concur at the centroid of the reference tetrahedron.\\
(c) The central tetrahedron is similar to the reference tetrahedron.\\
(d) The central centroid coincides with the reference centroid.\\
(e) The central circumcenter coincides with the reference Euler point.\\
(f) The central Monge point lies on the reference Euler line (at 2/3).\\
(g) The central Euler point lies on the reference Euler line (at 8/9).\\
(h) The reference circumcenter lies on the central Euler line (at 4).\\
(i) The reference Monge point lies on the central Euler line (at -2).
\end{theorem}

\begin{theorem}
For an arbitrary tetrahedron,
the normals at the circumcenters of each face
concur at the circumcenter of the reference tetrahedron.
\end{theorem}


\begin{theorem}
For an arbitrary tetrahedron,
the lines to the $r$-power points form a hyperbolic group.
These include the incenters, the centroids, and the symmedian points.
\end{theorem}


\textbf{Note.}
Results found for specific tetrahedra that are immediate consequences
of results in this section for arbitrary tetrahedra will not necessarily
be listed again below.

\begin{theorem}
For fixed $r$, the $2a^r+b^r+c^r$ points of an arbitrary tetrahedron
form a central tetrahedron 
that has the same centroid as the reference tetrahedron.
\end{theorem}


\section{Results found for Isosceles Tetrahedra}
\label{section:isosceles}

The following results were discovered and proven by our computer program.

\begin{theorem}
Consider an arbitrary center on each face of an isosceles tetrahedron.
(The same type of center is considered on each face.)
Then\\
(a) The cevians to these centers have the same length.\\
(b) The central tetrahedron is isosceles.\\
(c) The central tetrahedron has the same centroid as the reference tetrahedron.\\
(d) The cevians form a hyperbolic group.
\end{theorem}



\section{Results found for Circumscriptible Tetrahedra}
\label{section:circumscriptible}

The following results were discovered and proven by our computer program.


\begin{theorem}
For a circumscriptible tetrahedron (in which $a_i+b_i=t$, $i=1,2,3$),\\
(a) The cevians to the Gergonne points concur.\\
\hspace*{15pt} The 4th coordinate of the intersection point is\\
\hspace*{15pt} $(a_2+a_3-a_1)(a_3+a_1-a_2)(a_1+a_2-a_3)$.\\
(b) The cevians to the Nagel points concur.\\
\hspace*{15pt} The 4th coordinate of the intersection point is
$a_1+a_2+a_3-2t$.\\
\hspace*{15pt} This equals $\frac{S}{2}-S_4$ where $S_i$ is the sum of the edges at $A_i$
and $S=\sum S_i$.\\
(c) The Feuerbach points are coplanar.\\
(d) The normals at the incenters concur.\\
(e) The normals at the X40 points concur.\\
\hspace*{15pt} Note that X40 is collinear with the incenter and circumcenter.
\end{theorem}


\begin{theorem}
For a circumscriptible tetrahedron,
the lines to the following triangle centers form a hyperbolic group:\\
(a) Gergonne points (and their inverses)\\
(b) Nagel points (and their inverses)\\
(c) Mittenpunkts (and their inverses)\\
(d) X41 points (and their inverses)\\
(e) Feuerbach points (and their inverses).
\end{theorem}


\section{Results found for Isodynamic Tetrahedra}
\label{section:isodynamic}

The following results were discovered and proven by our computer program.


\begin{theorem}
For an isodynamic tetrahedron,\\
(a) The cevians to any power point concur.\\
\hspace*{15pt} The 4th coordinate of the intersection point is
$a_1a_2^{r+1}a_3$.\\
(b) The Feuerbach points are coplanar.\\
(c) The X44 points are coplanar.\\
(d) The Lemoine axes are coplanar.\\
(e) The circumcenter of the X76 points (3rd power point inverses)
coincides with the reference centroid.
\end{theorem}


\begin{theorem}
For an isodynamic tetrahedron,
the lines to the following triangle centers form a hyperbolic group:\\
(a) Spieker centers (and their inverses)\\
(b) X37 points (and their inverses)\\
(c) X38 points (and their inverses)\\
(d) Brocard midpoint (and their inverses)\\
(e) X42 points (and their inverses)\\
(f) X106 points\\
(g) X107 points\\
(h) X108 points\\
(i) X109 points\\
(j) X110 points\\
(k) X111 points
\end{theorem}




\section{Results found for Orthocentric Tetrahedra}
\label{section:orthocentric}

The following results were discovered and proven by our computer program.


\begin{theorem}
For an orthocentric tetrahedron (in which $a_i^2+b_i^2=t$, $i=1,2,3$),\\
(a) The cevians to the orthocenters concur.\\
\hspace*{15pt} The 4th coordinate of the intersection point is\\
\hspace*{15pt} $(a_2^2+a_3^2-a_1^2)(a_3^2+a_1^2-a_2^2)(a_1^2+a_2^2-a_3^2)$.\\
(b) The cevians to the isotomic conjugates of the orthocenters concur.\\
\hspace*{15pt} The 4th coordinate of the intersection point is
$a_1^2+a_2^2+a_3^2-2t$. This equals\\
\hspace*{15pt} $\frac{T}{2}-T_4$ where $T_i$ is the sum of the squares
of the edges at $A_i$
and $T=\sum T_i$.\\
(c) The centroid of the 9-point centers coincides with the reference centroid.\\
(d) The centroid of the orthocenters coincides with the reference Monge point.\\
(e) The circumcenter of the orthocenters lies on the reference Euler line.\\
(f) The Monge point of the orthocenters lies on the reference Euler line.\\
(g) The centroid of the X53 points coincide with the reference Monge point.
\end{theorem}

\begin{theorem}
For an orthocentric tetrahedron,\\
(a) The normals at the circumcenters concur.\\
(b) The normals at the centroids concur.\\
(c) The normals at the orthocenters concur.\\
(d) The normals at the nine point centers concur.\\
(e) The normals at the De Longchamps points concur.
\end{theorem}

Note that these five centers lie on the Euler line and have constant
ratio distances apart.

\begin{theorem}
For an orthocentric tetrahedron,
the lines to the following triangle centers form a hyperbolic group:\\
(a) circumcenters\\
(b) Crucial points (and their inverses)\\
(c) X25 points\\
(d) X48 points (and their inverses)
\end{theorem}


\section{Results found for Harmonic Tetrahedra}
\label{section:harmonic}

The following results were discovered and proven by our computer program.


\begin{theorem}
For a harmonic tetrahedron,\\
(a) The Feuerbach points are coplanar.\\
(b) The cevians to the X117 points (and their isotomic conjugates, the X102 points) concur.
\end{theorem}


\begin{theorem}
For a harmonic tetrahedron,
the lines to the following triangle centers form a hyperbolic group:\\
(a) X43 points (and their inverses)\\
(b) X102 points\\
(c) X117 points\\
\end{theorem}

\medskip
Note that $X43=1/b+1/c-1/a$, $X102=a(1/b+1/c-1/a)$, and $X117=(1/b+1/c-1/a)/a$.


\section{General Results about Concurrent Cevians}
\label{section:concurrentCevians}

The data collected by our program suggested (but did not prove) the following results.
Thus, independent proofs are needed.

\begin{lemma}
Let $P_1=(0,y_1,z_1,w_1)$ and $P_2=(x_2,0,z_2,w_2)$
be two points on faces 1 and 2 of the reference tetrahedron.
Then the condition that the lines $A_iP_i$, $i=1,2$ intersect (or be parallel) is
$$\frac{z_1}{w_1}=\frac{z_2}{w_2}.$$
\end{lemma}

\begin{proof}
From formulas 7 and 11, the condition is
$$
\begin{vmatrix}
0&y_1&z_1&w_1\\
x_2&0&z_2&w_2\\
1&0&0&0\\
0&1&0&0\\
\end{vmatrix}
= 0.
$$
which reduces to the formula claimed.
\end{proof}

\textbf{Geometric interpretation.}
Let $P_1$ and $P_2$ be points on the faces opposite vertices $A_1$ and $A_2$,
respectively, of tetrahedron $A_1A_2A_3A_4$. Then lines $A_1P_1$ and
$A_2P_2$ intersect if and only if
$$[P_1A_2A_3][P_2A_4A_1]=[P_1A_2A_4][P_2A_1A_3].$$
Here ``intersect'' also includes being parallel.

Figure \ref{fig:intersectingCevians} shows a top view of tetrahedron $A_1A_2A_3A_4$
with a point taken on faces 1 and 2.
The condition is that the product of the areas of the
yellow triangles equals the product of the areas of the
green triangles.

\begin{figure}[h!t]
\centering
\includegraphics[width=0.7\linewidth]{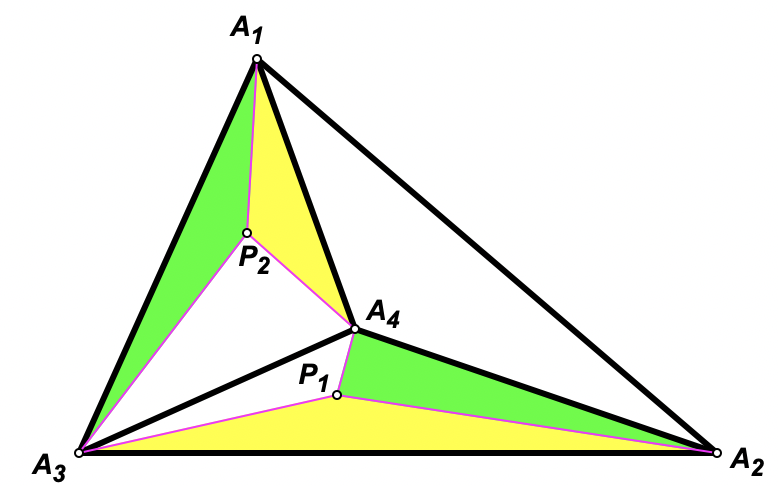}
\caption{$A_1P_1$ intersects $A_2P_2$ if product of yellow areas equals product of green areas}
\label{fig:intersectingCevians}
\end{figure}

\begin{lemma}
If $A_iP_i$, $i=1,2,3$, meet in pairs, then all three lines meet at a point.
\end{lemma}

\begin{proof}
Let the lines meet in pair at points $Q_1$, $Q_2$, $Q_3$. Then the plane
through $Q_1$, $Q_2$, and $Q_3$ contains the three lines and hence
the three vertices $A_1$, $A_2$, $A_3$. Thus $A_1$, $A_2$, $A_3$ would
lie on a plane other than the base plane, a contradiction.
\end{proof}

\begin{corollary}
Let $P_1=(0,y_1,z_1,w_1)$, $P_2=(x_2,0,z_2,w_2)$, $P_3=(x_3,y_3,0,w_3)$
be three points on faces 1, 2, and 3 of the reference tetrahedron.
Then the condition that the lines $A_iP_i$, $i=1,2,3$ concur (or be parallel) is
$$\frac{z_1}{w_1}=\frac{z_2}{w_2},\qquad\frac{y_1}{w_1}=\frac{y_3}{w_3},\qquad
\frac{x_2}{w_2}=\frac{x_3}{w_3}.$$
\end{corollary}

\begin{corollary}[\textbf{The Concurrence Condition}]
The condition for the concurrence of cevians to two centers (each with center function $F(a,b,c)$)
on faces
1 and 2 of the reference tetrahedron is
$$F(b_2,a_1,b_3)F(a_2,b_3,b_1)=F(a_1,b_3,b_2)F(b_1,a_2,b_3).$$
\end{corollary}

\begin{theorem}
If $P_1$ and $P_2$ are points on faces 1 and 2 of our reference tetrahedron
such that the cevians to $P_1$ and $P_2$ meet, then the cevians to the
isotomic conjugates of $P_1$ and $P_2$ meet.
\end{theorem}

\begin{proof}
In areal coordinates, the isotomic conjugate of $(x,y,z)$ is $(1/x,1/y,1/z)$.
The concurrence condition therefore becomes
$$\frac{1}{F(b_2,a_1,b_3)}\cdot \frac{1}{F(a_2,b_3,b_1)}=
\frac{1}{F(a_1,b_3,b_2)}\cdot \frac{1}{F(b_1,a_2,b_3)}$$
which is equivalent to the original condition.
\end{proof}

\begin{corollary}[{\cite[p.~139]{Altshiller}}]
If the four cevians to corresponding face centers concur,
then the four cevians to the isotomic conjugates of these centers also concur.
\end{corollary}

\begin{corollary}
If cevians to the triangle centers with center function $F(a,b,c)$ concur,
then so do cevians to the centers with center function $F(a,b,c)^r$ for any $r$.
\end{corollary}

\begin{theorem}
The centroid is the only triangle center with the property that in any
isosceles tetrahedron,
the cevians to these face centers concur.
\end{theorem}

\begin{proof}
Suppose $F(a,b,c)=af(a,b,c)$ is such a center function.
The algebraic condition for this to be true is obtained by
substituting $b_i=a_i$ in the concurrence condition to get
$F(a_1,a_2,a_3)^2=F(a_2,a_1,a_3)^2$. This implies that
\begin{equation}
\label{eq:E1}
F(a_1,a_2,a_3)=-F(a_2,a_1,a_3)
\end{equation}
or
\begin{equation}
\label{eq:E2}
F(a_1,a_2,a_3)=F(a_2,a_1,a_3)
\end{equation}
for all $a_1$, $a_2$, and $a_3$.
\par If condition (\ref{eq:E1}) holds, then we would have
$$F(a_1,a_2,a_3)=-F(a_2,a_1,a_3)$$
$$F(a_2,a_3,a_1)=-F(a_3,a_2,a_1)$$
$$F(a_3,a_1,a_2)=-F(a_1,a_3,a_2)$$
since the equality must be true for all values of its arguments.
Since $F(a,b,c)=F(a,c,b)$,
multiplying these three equations together yields $1=-1$,
a contradiction.
\par If condition (\ref{eq:E2}) holds, then we would have
$$\frac{F(a_1,a_2,a_3)}{F(a_2,a_3,a_1)}=1$$
or $$F(a_1,a_2,a_3):F(a_2,a_3,a_1):F(a_3,a_1,a_2)=1:1:1$$
so that $F$ represents the centroid.
\end{proof}

\begin{corollary}
The centroid is the only triangle center with the property that in any
tetrahedron, the cevians to these face centers concur.
\end{corollary}

\begin{proof}
Since the cevians concur for {\it any} tetrahedron, they
must surely concur for any isosceles tetrahedron.
But the previous theorem rules this possibility out.
\end{proof}

\void{
\smallskip
An independent proof can also be given.
The following proof is due to Peter Gilbert.

\begin{proof}
The concurrence condition can be written as
$$\frac{F(b_2,a_1,b_3)}{F(a_1,b_3,b_2)}
=\frac{F(b_1,a_2,b_3)}{F(a_2,b_3,b_1)}.$$
Changing variables, we can write this as
$$\frac{F(x,y,g)}{F(y,g,x)}
=\frac{F(p,q,g)}{F(q,g,p)}.$$
Fixing $p$ and $q$, we see that the right-hand side of this equation
depends only on $g$ and therefore we can define
$$G(g)=\frac{F(x,y,g)}{F(y,g,x)}.$$
For any $x$, $y$, and $z$, we have
$$
\begin{aligned}
F(x,y,z)&=F(x,z,y)=F(z,x,y)G(y)=F(z,y,x)G(y)\\
&=F(y,z,x)G(x)G(y)=F(y,x,z)G(x)G(y)\\
&=F(x,y,z)G(z)G(x)G(y).\\
\end{aligned}
$$

Hence
$$G(x)G(y)G(z)=1$$
for all $x$, $y$ and $z$.
Letting $z=y=x$, shows that $G(x)^3=1$ and so $G(x)=1$ for all $x$.
Hence
$$
\begin{aligned}
F(a,b,c):F(b,c,a):F(c,a,b)
&=\frac{F(a,b,c)}{F(a,b,c)}:\frac{F(b,c,a)}{F(a,b,c)}:\frac{F(c,a,b)}{F(a,b,c)}\\
&=1:G(c)^{-1}:G(b)\\
&=1:1:1\\
\end{aligned}
$$
and thus the center is the centroid.
\end{proof}
}

The following lemma is well known:

\begin{lemma}[Power Lemma]
If $f(x)$ is a nonzero function satisfying
$$f(xy)=f(x)f(y)$$
for all $x$ and $y$, then $f(x)=x^r$ for some constant $r$.
\end{lemma}

\begin{theorem}
\label{thm:power}
The power points are the only triangle centers with the property that in any
isodynamic tetrahedron, the cevians to these face centers concur.
\end{theorem}

\begin{proof}
The concurrence condition becomes
$$F(\frac{t}{a_2},a_1,\frac{t}{a_3})F(a_2,\frac{t}{a_3},\frac{t}{a_1})=
F(a_1,\frac{t}{a_3},\frac{t}{a_2})F(\frac{t}{a_1},a_2,\frac{t}{a_3})$$
for all $a_1$, $a_2$, $a_3$, and $t$.
We can write this as
$$\frac{F(\frac{t}{a_2},a_1,\frac{t}{a_3})}{F(a_1,\frac{t}{a_2},\frac{t}{a_3})}
=
\frac{F(\frac{t}{a_1},a_2,\frac{t}{a_3})}{F(a_2,\frac{t}{a_1},\frac{t}{a_3})}.$$
Since this is true for all $a_2$, it will be true if we replace $a_2$ by $t/a_2$
to get
$$\frac{F(a_2,a_1,\frac{t}{a_3})}{F(a_1,a_2,\frac{t}{a_3})}
=
\frac{F(\frac{t}{a_1},\frac{t}{a_2},\frac{t}{a_3})}{F(\frac{t}{a_2},\frac{t}{a_1},
\frac{t}{a_3})}.$$
Since $F$ is homogeneous, we have
$$\frac{F(a_2,a_1,\frac{t}{a_3})}{F(a_1,a_2,\frac{t}{a_3})}
=
\frac{F(\frac{1}{a_1},\frac{1}{a_2},\frac{1}{a_3})}{F(\frac{1}{a_2},
\frac{1}{a_1},\frac{1}{a_3})}.$$
Let $t=a_3x$ to get
$$\frac{F(a_2,a_1,x)}{F(a_1,a_2,x)}
=
\frac{F(\frac{1}{a_1},\frac{1}{a_2},\frac{1}{a_3})}{F(\frac{1}{a_2},
\frac{1}{a_1},\frac{1}{a_3})}.$$
The right-hand side is independent of $x$, and so we can define
\begin{equation}
\label{eq:P1}
G(a_1,a_2)=\frac{F(a_2,a_1,x)}{F(a_1,a_2,x)}.
\end{equation}
Thus
$$F(a,b,c)=G(a,b)F(b,a,c)=G(a,b)G(b,a)F(a,b,c)$$
and so
\begin{equation}
\label{eq:P2}
G(a,b)G(b,a)=1
\end{equation}
for all $a$ and $b$. Similarly,
$$
\begin{aligned}
F(a,b,c)&=G(a,b)F(b,a,c)=G(a,b)F(b,c,a)\\
&=G(a,b)G(b,c)F(c,b,a)=G(a,b)G(b,c)F(c,a,b)\\
&=G(a,b)G(b,c)G(c,a)F(a,c,b)=G(a,b)G(b,c)G(c,a)F(a,b,c)\\
\end{aligned}
$$
and so
\begin{equation}
\label{eq:P3}
G(a,b)G(b,c)G(c,a)=1
\end{equation}
for all $a$, $b$ and $c$.
Using the homogeneity of $F$ in equation (\ref{eq:P1}), we can divide all arguments by $a_2$ to get
$$G(a_1/a_2,1)=\frac{F(1,a_1/a_2,x/a_2)}{F(a_1/a_2,1,x/a_2)}=
\frac{F(a_2,a_1,x)}{F(a_1,a_2,x)}=G(a_1,a_2).$$
Thus, if we define
$$g(x)=G(x,1),$$
then we get analogs of equations (\ref{eq:P2}) and (\ref{eq:P3}):
$$g(\frac{a}{b})g(\frac{b}{a})=1$$
and
$$g(\frac{a}{b})g(\frac{b}{c})g(\frac{c}{a})=1.$$
Now
$$g(\frac{a}{b})g(\frac{b}{c})={1/g(\frac{c}{a})}=g(\frac{a}{c}).$$
Let $x=a/b$ and $y=b/c$ to get
$$g(x)g(y)=g(xy)$$
for all $x$ and $y$.
By the Power Lemma, we must have $g(x)=x^r$ for some $r$.
Hence
$$
\begin{aligned}
F(a,b,c):F(b,c,a):F(c,a,b)&=\frac{F(a,b,c)}{F(a,b,c)}:\frac{F(b,c,a)}{F(a,b,c)}:
\frac{F(c,a,b)}{F(a,b,c)}\\
&=1:G(b,a):G(c,a)\\
&=1:g(\frac{b}{a}):g(\frac{c}{a})\\
&=1:\Bigl(\frac{b}{a}\Bigr)^r:\Bigl(\frac{c}{a}\Bigr)^r\\
&=a^r:b^r:c^r\\
\end{aligned}
$$
and thus the center is a power point.
\end{proof}

\begin{theorem}
Suppose the edges of a tetrahedron satisfy the condition
$h(a_i)+h(b_i)=t$, $i=1,2,3$, for some function $h(x)$ and some constant, $t$.
Then the cevians to the face centers with (areal) center function
$[h(b)+h(c)-h(a)]^r$ concur for any $r$.
\end{theorem}

\begin{proof}
The concurrence condition is
$$[h(a_1)+h(b_3)-h(b_2)]^r[h(b_3)+h(b_1)-h(a_2)]^r$$
$$\qquad=
[h(b_3)+h(b_2)-h(a_1)]^r[h(a_2)+h(b_3)-h(b_1)]^r.$$
This is equivalent to
$$[h(a_1)+h(b_3)+h(a_2)-t]^r[h(b_3)+h(b_1)+h(b_2)-t]^r$$
$$\qquad=
[h(b_3)+h(b_2)+h(b_1)-t]^r[h(a_2)+h(b_3)+h(a_1)-t]^r$$
which is easily seen to be an identity.
\end{proof}

\begin{corollary}
In a circumscriptible tetrahedron, the cevians to the face centers
with areal center function $(b+c-a)^r$ concur.
This includes the Gergonne point and its isotomic conjugate, the Nagel point.
\end{corollary}

\begin{corollary}
In an orthocentric tetrahedron, the cevians to the face centers
with areal center function $(b^2+c^2-a^2)^r$ concur.
This includes the orthocenter and its isotomic conjugate.
\end{corollary}

\begin{corollary}
In a harmonic tetrahedron, the cevians to the face centers
with areal center function $(1/b+1/c-1/a)^r$ concur.
This includes the X117 point and its isotomic conjugate, the X102 point.
\end{corollary}

\begin{corollary}
In an isodynamic tetrahedron, the cevians to the power points concur.
This includes the incenter, centroid, symmedian point and their isogonal
and isotomic conjugates.
\end{corollary}

\begin{proof}
Take $h(x)=\log x$ in the previous theorem.
\end{proof}

\begin{conjecture}
Suppose the edges of a tetrahedron satisfy the condition
$h(a_i)+h(b_i)=t$, $i=1,2,3$, for some power function $h(x)=x^n$ and some constant, $t$.
If cevians to four corresponding face centers are concurrent, then
the center function for these face centers must be of the form
$[h(b)+h(c)-h(a)]^r$ for some $r$.
\end{conjecture}

\begin{theorem}
If cevians to the points $h(b)+h(c)-h(a)$ concur, then the edges of the tetrahedron
satisfy $h(a_i)+h(b_i)=t$, for $i=1,2,3$.
\end{theorem}

\begin{proof}
The concurrency condition (for centers on faces 1 and 2) becomes
$$[h(a_1)+h(b_3)-h(b_2)][h(b_3)+h(b_1)-h(a_2)]$$
$$\qquad=[h(b_3)+h(b_2)-h(a_1)][h(a_2)+h(b_3)-h(b_1)].$$
Simple algebra transforms this into the equation
$$h(b_3)[h(a_1)+h(b_1)]=h(b_3)[h(a_2)+h(b_2)]$$
from which we conclude that
$$h(a_1)+h(b_1)=h(a_2)+h(b_2).$$
By symmetry, analogous results are true for any two faces of the reference
tetrahedron, so $h(a_i)+h(b_i)$ is constant for all $i$.
\end{proof}

\begin{corollary}[{\cite[p.~299]{Altshiller}}]
If cevians to the Nagel points concur, then the tetrahedron is isodynamic.
\end{corollary}

\begin{corollary}[{\cite[p.~299]{Altshiller}}]
If cevians to the Gergonne points concur, then the tetrahedron is isodynamic.
\end{corollary}

\begin{corollary}
If cevians to the orthocenters concur, then the tetrahedron is orthocentric.
\end{corollary}

\begin{corollary}
If cevians to the $(1/b+1/c-1/a)$ centers concur, then the tetrahedron is harmonic.
\end{corollary}

\begin{corollary}
For a fixed $r\neq 0$, if cevians to the $r$-power points concur, then the tetrahedron
is isodynamic.
\end{corollary}

This generalizes proposition 841 of \cite{Altshiller} (which was for the Symmedian point only).

\begin{proof}
The $a^r$ centers are the same as the $a^rb^rc^r/a^r$ centers, so the result follows
by taking $h(x)=\log x$.
\end{proof}

\section{General Results about Hyperbolic Lines}
\label{section:hyperbolicLines}

\begin{theorem}
Let $P_1=(0,y_1,z_1,w_1)$, $P_2=(x_2,0,z_2,w_2)$, and $P_3=(x_3,y_3,0,w_3)$
be three points on faces 1, 2, and 3 of the reference tetrahedron.
Then the condition that there is a line through vertex $A_4$ that meets
all three of the lines $A_iP_i$, $i=1,2,3$ is
$$z_1x_2y_3=y_1z_2x_3.$$
This line is called a spear line.
The spear line meets face $A_1A_2A_3$ at the point $(x_3y_1,y_1y_3,y_3z_1,0)$.
(This point is known as the spear trace.)
\end{theorem}

\begin{proof}
This result was found by computer but could easily be carried out by hand.
Formula 16 gives us the equation of the plane, $E$, through $A_4$ and $A_3P_3$.
Any spear line must clearly lie in this plane.
Formula 18 determines the point, $Q_1$, where line $A_1P_1$ meets plane $E$.
Then $A_4Q_1$ must be the desired spear line.
Similarly, we can find the point $Q_2$, where line $A_2P_2$ meets plane $E$.
Then $A_4Q_2$ must also be the desired spear line.
Thus the condition is that points $A_4$, $Q_1$, and $Q_2$ colline.
Formula 3 gives us this condition. Upon simplifying the result, the
computer came up with $z_1x_2y_3=y_1z_2x_3$ as the algebraic condition.
We can then find the intersection of
the common line $A_4Q_1Q_2$ and the plane $A_1A_2A_3$ to get the coordinates
of the spear trace.
\end{proof}

\textbf{Geometric interpretation.}\\
Let $P_1$, $P_2$, and $P_3$ be points on the faces opposite vertices $A_1$,
$A_2$, and $A_3$,
respectively, of tetrahedron $A_1A_2A_3A_4$.
Then there is a line through $A_4$ that meets
lines $A_1P_1$, $A_2P_2$ and $A_3P_3$ if and only if
$$[P_1A_2A_4][P_2A_3A_4][P_3A_1A_4]=
[P_1A_3A_4][P_2A_1A_4][P_3A_2A_4]$$
where $[XYZ]$ denotes the area of triangle $XYZ$.

Figure \ref{fig:hyperbolicCondition}a shows a top view of tetrahedron $A_1A_2A_3A_4$.
Figure \ref{fig:hyperbolicCondition}b then shows a point taken on faces 1, 2, and 3.
The condition is that the product of the areas of the
yellow triangles equals the product of the areas of the
green triangles.

\begin{figure}[h!t]
\centering
\includegraphics[width=1\linewidth]{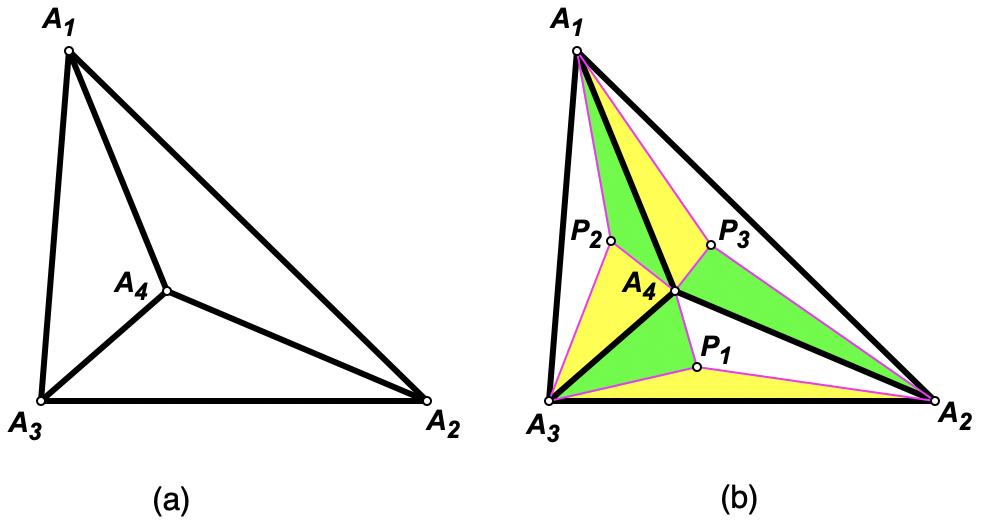}
\caption{There is a line through $A_4$ that meets $A_iP_i$, $i=1,2,3$ if and only if product of yellow areas equals product of green areas.}
\label{fig:hyperbolicCondition}
\end{figure}

We say that a center function $F(a,b,c)=af(a,b,c)$ is a hyperbolic center
function if the cevians from the vertices of a tetrahedron to these
centers on the opposite faces form a hyperbolic group.

\begin{proposition}[{\cite[p.~11]{Altshiller}}]
If four given mutually skew lines passing through
the four vertices of a tetrahedron are such that through each vertex
it is possible to draw a spear line meeting the three lines
passing through the remaining three vertices, then the four given
lines form a hyperbolic group.
\end{proposition}

\begin{corollary}[Hyperbolic Condition]
The condition that $F(a,b,c)=af(a,b,c)$ be a hyperbolic center function is:
$$F(b_2,a_1,b_3)F(b_3,b_1,a_2)F(b_1,a_3,b_2)=
F(b_3,b_2,a_1)F(b_1,a_2,b_3)F(b_2,b_1,a_3).$$
\end{corollary}

This follows from the coordinates for the corresponding center
on each face found in section \ref{section:faceCenters}
and the preceding Proposition.
Note that symmetry conditions imply that we need only find the
condition that one spear line exists (instead of all 4).

\begin{corollary}
In an isosceles tetrahedron, all center functions are hyperbolic.
\end{corollary}

\begin{proof}
When $b_i=a_i$, the hyperbolic condition becomes
$$F(a_2,a_1,a_3)F(a_3,a_1,a_2)F(a_1,a_3,a_2)=
F(a_3,a_2,a_1)F(a_1,a_2,a_3)F(a_2,a_1,a_3)$$
which is clearly an identity since 
a center function $F(a,b,c)$ is symmetric in $b$ and $c$.
\end{proof}

\begin{theorem}
If  $F(a,b,c)$ is a hyperbolic center function,
then so is $a^rF(a,b,c)^q$.
\end{theorem}

\begin{proof}
The hyperbolic condition becomes
$$b_2^{r}F(b_2,a_1,b_3)^qb_3^{r}F(b_3,b_1,a_2)^qb_1^{r}F(b_1,a_3,b_2)^q$$
$$\qquad=
b_3^{r}F(b_3,b_2,a_1)^qb_1^{r}F(b_1,a_2,b_3)^qb_2^{r}F(b_2,b_1,a_3)^q$$
which is an immediate consequence of the original condition.
\end{proof}

\begin{theorem}
If cevians to corresponding centers on each face of a tetrahedron form a hyperbolic group,
then so do cevians to the isotomic conjugates of those centers.
\end{theorem}

\begin{proof}
Since in areal coordinates, the isotomic conjugate of a
center $(x,y,z)$ is $(1/x,1/y,1/z)$,
the hyperbolic condition becomes
$$\frac{1}{F(b_2,a_1,b_3)}\frac{1}{F(b_3,b_1,a_2)}\frac{1}{F(b_1,a_3,b_2)}=
\frac{1}{F(b_3,b_2,a_1)}\frac{1}{F(b_1,a_2,b_3)}\frac{1}{F(b_2,b_1,a_3)}$$
which is equivalent to the original condition.
\end{proof}

\begin{corollary}[{\cite[p.~332]{Altshiller}}]
If cevians to corresponding centers on each face of a tetrahedron form a hyperbolic group,
then so do cevians to the isogonal conjugates of those centers.
\end{corollary}

\begin{proof}
This is because in areal coordinates, the isogonal conjugate of a center $(x,y,z)$
is $(a^{-2}/x,b^{-2}/y,c^{-2}/z)$. Thus the result follows from the
previous two theorems.
\end{proof}

\begin{theorem}
Suppose the edges of a tetrahedron satisfy the condition
$h(a_i)+h(b_i)=t$, $i=1,2,3$, for some function $h(x)$ and some constant, $t$.
Then the cevians to the face centers with center function
$[h(b)+h(c)-h(a)]^q$ form a hyperbolic group.
\end{theorem}

\begin{proof}
The hyperbolic condition becomes
$$[h(a_1)+h(b_3)-h(b_2)]^q[h(b_1)+h(a_2)-h(b_3)][h(a_3)+h(b_2)-h(b_1)]^q=$$
$$[h(b_2)+h(a_1)-h(b_3)]^q[h(a_2)+h(b_3)-h(b_1)][h(b_1)+h(a_3)-h(b_2)]^q$$
which immediately follows from
$$[h(a_1)+h(b_3)+h(a_2)-t]^q[h(b_1)+h(a_2)+h(a_3)-t]^q[h(a_3)+h(b_2)+h(a_1)-t]^q=$$
$$[h(b_2)+h(a_1)+h(a_3)-t]^q[h(a_2)+h(b_3)+h(a_1)-t]^q[h(b_1)+h(a_3)+h(a_2)-t]^q$$
which is identically true.
\end{proof}

\begin{corollary}
In a circumscriptible tetrahedron, the cevians to the face centers
with center function $a^r(b+c-a)$ form a hyperbolic group.
This includes the Gergonne point, the Nagel point, the Mittenpunkt,
the X41 point, and their isogonal and isotomic conjugates.
\end{corollary}

\begin{corollary}
In an orthocentric tetrahedron, the cevians to the face centers
with center function $a^r(b^2+c^2-a^2)$ form a hyperbolic group.
This includes the orthocenter, circumcenter, the crucial point,
the X25 point, the X48 point, and their isogonal and isotomic conjugates.
\end{corollary}

\begin{corollary}
In a harmonic tetrahedron, the cevians to the face centers
with center function $a^r(1/b+1/c-1/a)$ form a hyperbolic group.
This includes the
X43 point, the X102 point, the X117 point,
and their isogonal and isotomic conjugates.
\end{corollary}

\begin{theorem}
The power points are the only triangle centers with the property that in any tetrahedron,
the cevians to these face centers form a hyperbolic group.
\end{theorem}

\begin{proof}  
Let $F(a,b,c)=af(a,b,c)$ be a hyperbolic center function. Then $F$ must
satisfy the functional equation
$$F(b_2,a_1,b_3)F(b_3,b_1,a_2)F(b_1,a_3,b_2)=
F(b_3,b_2,a_1)F(b_1,a_2,b_3)F(b_2,b_1,a_3).$$
This equation must be true
for all values of $a_1$, $a_2$, $a_3$, $b_1$, $b_2$, and $b_3$.
Rewriting this as
\begin{equation}
\label{eq:G1}
\frac{F(b_2,a_1,b_3)}{F(b_3,b_2,a_1)}=\frac{F(b_1,a_2,b_3)}{F(b_3,b_1,a_2)}
\cdot \frac{F(b_2,b_1,a_3)}{F(b_1,a_3,b_2)}
\end{equation}
shows that $F(b_2,a_1,b_3)/ F(b_3,b_2,a_1)$ is independent of $a_1$,
so we may define
\begin{equation}
\label{eq:G2}
G(b_2,b_3)=\frac{F(b_2,a_1,b_3)}{F(b_3,b_2,a_1)}.
\end{equation}
Substituting this in equation (\ref{eq:G1}) yields
$$G(b_2,b_3)=\frac{G(b_1,b_3)}{G(b_1,b_2)}$$
for all $b_1$, $b_2$, and $b_3$.
If we then let
$$H(z)=G(b_1,z)$$
we find that
$$G(x,y)=\frac{H(y)}{H(x)}$$
for all $x$ and $y$. From the definition of $G$, (equation \ref{eq:G2}), we get
$$\frac{H(y)}{H(x)}=\frac{F(x,z,y)}{F(y,x,z)}$$
for all $x$, $y$, and $z$.
Since $F$ is homogeneous, this implies that
$$\frac{H(ty)}{H(tx)}=\frac{H(y)}{H(x)}$$
for all $x$, $y$, and $t$. Letting $K(y)=H(y)/H(x)$ gives
$$\frac{K(ty)}{K(y)}=\frac{H(ty)}{H(y)}=\frac{H(tx)}{H(x)}=K(tx).$$
Letting $x=1$ shows that
$$K(ty)=K(t)K(y)$$
for all $t$ and $y$, so by the Power Lemma we have $K(x)=x^r$ for some constant $r$.
Denoting $H(1)$ by $c$, we have
$$H(x)=cK(x)=cx^r.$$
The coordinates for the center are
$$
\begin{aligned}
F(x,y,z):F(y,z,x):F(z,x,y)&=
\frac{F(x,y,z)}{F(x,y,z)}:\frac{F(y,z,x)}{F(x,y,z)}:\frac{F(z,x,y)}{F(x,y,z)}\\
&=1:\frac{H(x)}{H(y)}:\frac{H(x)}{H(z)}\\
&=\frac{1}{H(x)}:\frac{1}{H(y)}:\frac{1}{H(z)}\\ &=x^r:y^r:z^r\\
\end{aligned}
$$
and thus the center is a power point.
\end{proof}

\begin{theorem}
In an isodynamic tetrahedron, the center function $a^rg(b,c)$
is hyperbolic for any symmetric homogeneous function $g$.
\end{theorem}

\begin{proof}
The hyperbolic condition becomes
$$b_2^rg(a_1,b_3)b_3^rg(b_1,a_2)b_1^rg(a_3,b_2)=
b_3^rg(b_2,a_1)b_1^rg(a_2,b_3)b_2^rg(b_1,a_3).$$
Since the tetrahedron is isodynamic, this is equivalent to
$$b_2^rg(a_1,t/a_3)b_3^rg(t/a_1,a_2)b_1^rg(a_3,t/a_2)=
b_3^rg(t/a_2,a_1)b_1^rg(a_2,t/a_3)b_2^rg(t/a_1,a_3).$$
Using the fact that
$g$ is homogeneous, we see that this is equivalent to
$$b_2^rg(a_1a_3,t)b_3^rg(t,a_1a_2)b_1^rg(a_2a_3,t)=
b_3^rg(t,a_1a_2)b_1^rg(a_2a_3,t)b_2^rg(t,a_1a_3)$$
which is easily seen to be an identity since $g$ is symmetric.
\end{proof}

\begin{corollary}
In an isodynamic tetrahedron, the cevians to the face centers
with center function $a^r(b^q+c^q)^{\pm 1}$ form a hyperbolic group.
This includes the Spieker center, the Brocard midpoint,
the X37, X38, and X42 points, and their isogonal and isotomic
conjugates.
\end{corollary}

\begin{conjecture}
The center functions $a^rg(b,c)$, where $g$ is a symmetric homogeneous function, and
$r$ is arbitrary,
are the only hyperbolic center functions for an isodynamic tetrahedron.
\end{conjecture}

\section{General Planarity Results}
\label{section:planarity}

\begin{theorem}
No triangle center has the property that in any isosceles tetrahedron,
the corresponding face centers are coplanar.
\end{theorem}

\begin{proof}
The algebraic condition for planarity (found by computer) is one of
\begin{equation}
\label{eq:D1}
F(a_1,a_2,a_3)+F(a_2,a_3,a_1)=F(a_3,a_1,a_2)
\end{equation}
\begin{equation}
\label{eq:D2}
F(a_2,a_3,a_1)+F(a_3,a_1,a_2)=F(a_1,a_2,a_3)
\end{equation}
\begin{equation}
\label{eq:D3}
F(a_3,a_1,a_2)+F(a_1,a_2,a_3)=F(a_2,a_3,a_1)
\end{equation}
\begin{equation}
\label{eq:D4}
F(a_1,a_2,a_3)+F(a_2,a_3,a_1)+F(a_3,a_1,a_2)=0.
\end{equation}
If condition (\ref{eq:D4}) holds, then the center $(\alpha,\beta,\gamma)$
in areal coordinates would satisfy $\alpha+\beta+\gamma=0$, a contradiction.

\smallskip
If condition (\ref{eq:D1}) holds, then since it must be valid for
all variables $a_1$, $a_2$, $a_3$, we see that conditions (\ref{eq:D2}) and (\ref{eq:D3})
would also hold. Adding these three equations yields condition (\ref{eq:D4})
which we have already seen yields a contradiction.

\smallskip
The same contradiction is
reached if we assume (\ref{eq:D2}) or (\ref{eq:D3}) holds. 
\end{proof}

\begin{corollary}
No triangle center has the property that in any tetrahedron,
the corresponding face centers are coplanar.
\end{corollary}

The following condition was found by our computer program.

\begin{theorem}
The Feuerbach points on the faces of a tetrahedron are coplanar if
the edges of the tetrahedron satisfy the following condition:
$$
\begin{vmatrix}
a_1+b_1&a_2+b_2&a_3+b_3\\
a_1b_1&a_2b_2&a_3b_3\\
1&1&1\\
\end{vmatrix}
=0.$$
\end{theorem}

In particular, the Feuerbach points are coplanar for circumscriptible, isodynamic,
and harmonic tetrahedra. The Feuerbach points would also be coplanar for
tetrahedra that satisfy other simple relationships, for example, ones
in which $a_ib_i+a_i+b_i$ were constant.

\section{General Results about Concurrent Normals}
\label{section:concurrentNormals}

\begin{theorem}
Let $P_1$ and $P_2$ be any two points on faces 1 and 2 of tetrahedron
$A_1A_2A_3A_4$, respectively.  If normals to the faces at $P_1$ and $P_2$
concur, then
$$(P_2A_3)^2+(P_1A_4)^2=(P_2A_4)^2+(P_1A_3)^2.$$
\end{theorem}

\begin{proof}
From right triangles $PP_2A_3$, $PP_1A_3$, $PP_2A_4$, and $PP_1A_4$, we have
$$(A_3P_2)^2+(PP_2)^2=(A_3P)^2=(A_3P_1)^2+(PP_1)^2$$
and
$$(A_4P_2)^2+(PP_2)^2=(A_4P)^2=(A_4P_1)^2+(PP_1)^2.$$
Subtracting one equation from the other gives us the desired result.
\end{proof}

\medskip
The converse is also true.

\begin{theorem}[Tabov \cite{Tabov}]
Let $P_1$ and $P_2$ be any two points on faces 1 and 2 of tetrahedron
$A_1A_2A_3A_4$, respectively.
If
$$(P_2A_3)^2+(P_1A_4)^2=(P_2A_4)^2+(P_1A_3)^2,$$
then normals to the faces at $P_1$ and $P_2$ concur.
\end{theorem}

\section{Miscellaneous Conjectures}
\label{section:conjectures}

The following conjectures are backed up by the data,
but I don't have formal proofs.

\begin{conjecture}
If the central tetrahedron is isosceles, then the reference tetrahedron
is isosceles.
\end{conjecture}

\begin{conjecture}
If the central tetrahedron is regular, then the reference tetrahedron
is regular.
\end{conjecture}

\begin{conjecture}
If the central tetrahedron is similar to the reference tetrahedron,
then the center must be the centroid.
\end{conjecture}

\begin{conjecture}
If the cevians to the corresponding face centers have the same length,
then the tetrahedron is isosceles.
\end{conjecture}

\appendix
\section{Center Functions and Trilinear Coordinates}
\label{section:trilinearCoordinates}

In this appendix, we give the basic information about center functions
and trilinear coordinates that the reader needs to know.

Before giving the definition of a triangle center, let us review the
concept of trilinear coordinates.  If $ABC$ is a fixed reference triangle
in the plane (with sides of lengths $a$, $b$, and $c$), and
if $P$ is an arbitrary point in the plane, then the trilinear
coordinates of $P$ are $(\alpha,\beta,\gamma)$
where $\alpha$, $\beta$, $\gamma$ are the signed distances from $P$
to the sides $BC$, $CA$, $AB$, respectively.
The three coordinates satisfy the condition
\begin{equation}
\label{eq:A1}
a\alpha+b\beta+c\gamma=2K
\end{equation}
where $K$ is the area of $\triangle ABC$.
If $\alpha$, $\beta$, and $\gamma$ are any three real numbers
(with $a\alpha+b\beta+c\gamma\neq 0$), then there is a unique point $P$ in the plane
whose trilinear coordinates are proportional to $\alpha:\beta:\gamma$.
Thus $(\alpha,\beta,\gamma)$ may be considered to be the
trilinear coordinates of $P$ even if condition (\ref{eq:A1}) is not satisfied.
If condition (1) is satisfied, then we refer to the coordinates
as exact trilinear coordinates.

A {\it center function} is a nonzero function $f(a,b,c)$
that is homogeneous in $a$, $b$, and $c$ and symmetric in $b$ and $c$.
In other words, a center function must satisfy the following two conditions
for all $a$, $b$, $c$, $t$, and some integer $r$:
$$f(ta,tb,tc)=t^rf(a,b,c)\leqno(\rm C1)$$
$$f(a,c,b)=f(a,b,c)\leqno(\rm C2)$$
A {\it center} is an ordered triple $\alpha: \beta: \gamma$
given by
$$\alpha=f(a,b,c),\qquad \beta=f(b,c,a),\qquad \gamma=f(c,a,b)$$
for some center function $f(a,b,c)$. See \cite{KimberlingA} for more details.
If $f$ is a polynomial, then the center is called a polynomial center.

If $P=(\alpha, \beta, \gamma)$, then the point
$(\alpha^{-1},\beta^{-1},\gamma^{-1})$ is denoted by $P^{-1}$
and is called the isogonal conjugate of $P$.  For the
geometric meaning of isogonal conjugates, consult \cite{CollegeGeom}.

More information about center functions and trilinear coordinates
can be found in \cite{KimberlingA}, \cite{KimberlingB}, and \cite{KimberlingC}.

\section{Areal Coordinates}
\label{section:arealCoordinates}

In this appendix, we give the basic information about areal coordinates that the reader needs to know.

Let $ABC$ be a fixed reference triangle in the plane.
If $P$ is an arbitrary point in the plane of $\triangle ABC$, then the areal
coordinates of $P$ are given by $(x,y,z)$
where
$$
\begin{aligned}
x&=[PBC]/[ABC],\\
y&=[PCA]/[ABC],\\
z&=[PAB]/[ABC].\\
\end{aligned}
$$
The three areal coordinates satisfy the condition
\begin{equation}
\label{eq:AC1}
x+y+z=1.
\end{equation}
Areal coordinates are also known as barycentric coordinates.

If $x$, $y$, and $z$ are any three real numbers
(with $x+y+z\neq 0$), then there is a unique point $P$ in the plane
whose areal coordinates are proportional to $x:y:z$.
Thus $(x,y,z)$ may be considered to be the
areal coordinates of $P$ even if condition (\ref{eq:AC1}) is not satisfied.
The three coordinates are proportional to the areas formed by
$P$ and the sides of the reference triangle $ABC$.
If condition (\ref{eq:AC1}) is satisfied, then we refer to the coordinates
as exact areal coordinates.

Areal coordinates, $(x,y,z)$,
can be transformed to trilinear coordinates, $(\alpha,\beta,\gamma)$,
and vice versa, by the following formulas:
$$\alpha=x/a,\qquad \beta=y/b,\qquad \gamma=z/c$$
where $a$, $b$, and $c$ are the lengths of the sides of the reference triangle.

If $P=(x, y, z)$, then the point
$(x^{-1},y^{-1},z^{-1})$ is denoted by $P\T$
and is called the isotomic conjugate of $P$.  For the
geometric meaning of isotomic conjugates, consult \cite{CollegeGeom}.
In terms of trilinear coordinates, the isotomic conjugate
of $(\alpha,\beta,\gamma)$ is
$(\alpha^{-1}/a^2,\beta^{-1}/b^2,\gamma^{-1}/c^2)$.

More information about areal coordinates can be found in \cite{Capitan}
and \cite{Grozdev}.

\section{Tetrahedral Coordinates}
\label{section:tetrahedralCoordinates}

In this appendix, we give the basic information about tetrahedral coordinates that the reader needs to know.

The 3-dimensional analog of areal coordinates are tetrahedral
coordinates. Let $[T]$ denote the volume of a tetrahedron $T$.
Let $A_1A_2A_3A_4$ be a fixed reference tetrahedron in space
with edges of lengths $A_2A_3=a_1$, $A_3A_1=a_2$, $A_1A_2=a_3$,
$A_1A_4=b_1$, $A_2A_4=b_2$, and $A_3A_4=b_3$, so that
the edges of lengths $a_i$ and $b_i$ are opposite each other,
$i=1,2,3$.

If $P$ is an arbitrary point in space, then the tetrahedral
coordinates of $P$ are given by $(x_1,x_2,x_3,x_4)$
where
$$
\begin{aligned}
x_1&=[PA_2A_3A_4]/V,\\
x_2&=[PA_1A_3A_4]/V,\\
x_3&=[PA_1A_2A_4]/V,\\
x_4&=[PA_2A_3A_4]/V,\\
\end{aligned}
$$
where $V=[A_1A_2A_3A_4]$ is the volume of the
reference tetrahedron.
The four tetrahedral coordinates satisfy the condition
\begin{equation}
\label{eq:TC1}
x_1+x_2+x_3+x_4=1.
\end{equation}

If $x_i$, $i=1,2,3,4$ are any four real numbers
(with nonzero sum), then there is a unique point $P$ in space
whose tetrahedral coordinates are proportional to $x_1:x_2:x_3:x_4$.
Thus $(x_1,x_2,x_3,x_4)$ may be considered to be the
tetrahedral coordinates of $P$ even if condition (\ref{eq:TC1}) is not satisfied.
The four coordinates are proportional to the volumes formed by
$P$ and the faces of the reference tetrahedron $ABCD$.
If condition (\ref{eq:TC1}) is satisfied, then we refer to the coordinates
as exact tetrahedral coordinates.

For more information about tetrahedral coordinates, see \cite{Carr} and \cite{Frost}.

\section{Formulas}
\label{section:formulas}

\def\formula #1:{\smallskip\noindent \underbar{Formula #1}:}

In this appendix, we collect together formulas about tetrahedral coordinates
that are needed in this paper.
All points are given using exact tetrahedral coordinates.

\bigskip

\textbf{POINTS}

\formula 1:
Coordinates of a point (\cite[p.~65]{Frost}):
$$(x,y,z,w)$$
where $x+y+z+w=1$.
The coordinates are proportional to the volumes of the tetrahedra
formed by the point and the faces of the reference tetrahedron.
If a more symmetric notation is needed, we will use
the alternate form $(X_1,X_2,X_3,X_4)$.

\formula 2:
Condition for 4 points $(x_i,y_i,z_i,w_i)$, $i=1,2,3,4$
to be coplanar:
$$
\begin{vmatrix}
x_1&y_1&z_1&w_1\\
x_2&y_2&z_2&w_2\\
x_3&y_3&z_3&w_3\\
x_4&y_4&z_4&w_4\\
\end{vmatrix}
=0$$
or equivalently,
$$
\begin{vmatrix}
x_2-x_1&y_2-y_1&z_2-z_1\\
x_3-x_1&y_3-y_1&z_3-z_1\\
x_4-x_1&y_4-y_1&z_4-z_1\\
\end{vmatrix}
=0.$$

\formula 3:
Condition for 3 points $(x_i,y_i,z_i,w_i)$, $i=1,2,3$
to be collinear:
$$\frac{x_2-x_1}{x_3-x_1}=\frac{y_2-y_1}{y_3-y_1}=\frac{z_2-z_1}{z_3-z_1}=
\frac{w_2-w_1}{w_3-w_1}.$$

\formula 4:
Square of distance between points $(x_1,y_1,z_1,w_1)$ and $(x_2,y_2,z_2,w_2)$
(\cite[p.~66]{Frost}):
$$
(y_1-y_2)(z_2-z_1)a_1^2+
(z_1-z_2)(x_2-x_1)a_2^2+
(x_1-x_2)(y_2-y_1)a_3^2+
$$
$$
(x_1-x_2)(w_2-w_1)b_1^2+
(y_1-y_2)(w_2-w_1)b_2^2+
(z_1-z_2)(w_2-w_1)b_3^2.
$$
This can be written in the symmetrical form:
$$-\sum_{i,j}d_{i,j}^2(X_i-X'_i)(X_j-X'_j)$$
representing the square of the distance from $(X_1,X_2,X_3,X_4)$ to
$(X'_1,X'_2,X'_3,X'_4)$.  In this symmetrical notation,
the summation symbol
$$\sum_{i,j}\qquad {\rm means}\qquad \sum_{\scriptstyle i,j=1\atop \scriptstyle i<j}^4$$
and $d_{i,j}$ represents the length of edge $A_iA_j$, so that
$$d_{2,3}=a_1,\quad d_{1,3}=a_2,\quad d_{1,2}=a_3,\quad d_{1,4}=b_1,\quad 
d_{2,4}=b_2,\quad d_{3,4}=b_3.$$

\goodbreak
\textbf{LINES}

\formula 5:
General equation of a straight line
(\cite[p.~67]{Frost}):
$$\frac{x-x_0}{K}=\frac{y-y_0}{L}=\frac{z-z_0}{M}=\frac{w-w_0}{N}$$
where $K+L+M+N=0$. The line passes through the point $(x_0,y_0,z_0,w_0)$.
The quadruple, $(K,L,M,N)$, represents the direction of the line
and is called the direction vector.
Two lines are parallel if and only if they have the same direction vector
(or a multiple thereof).
Note that some of $K,L,M,N$ may be 0 because the condition
$(x-x_0)/K=(y-y_0)/L$ is really an abbreviation for
$(x-x_0)L=(y-y_0)K$ which does not involve any possible divisions by 0.

\formula 6: Parametric equation of a straight line:
$$(x_0+Kt,y_0+Lt,z_0+Mt,w_0+Nt).$$
This formula yields all the points on the line through $(x_0,y_0,z_0,w_0)$
with direction vector $(K,L,M,N)$ as $t$ varies
through the real numbers.

\formula 7:
Equation of line through $(x_1,y_1,z_1,w_1)$ and $(x_2,y_2,z_2,w_2)$:
$$
\frac{x-x_1}{x_2-x_1}=\frac{y-y_1}{y_2-y_1}=\frac{z-z_1}{z_2-z_1}=\frac{w-w_1}{w_2-w_1}.
$$

\formula 8:
Coordinates of point that divides the line joining points
$(x_1,y_1,z_1,w_1)$ and $(x_2,y_2,z_2,w_2)$ in the ratio $\mu:\lambda$
(\cite[p.~65]{Frost}):
$$(\frac{\lambda x_1+\mu x_2}{\lambda+\mu},\frac{\lambda y_1+\mu y_2}{\lambda+\mu},
\frac{\lambda z_1+\mu z_2}{\lambda+\mu},\frac{\lambda w_1+\mu w_2}{\lambda+\mu}).$$

\formula 9:
Condition for two lines
$(x-x_i)/ K_i=(y-y_i)/ L_i=(z-z_i)/M_i=(w-w_i)/N_i$, $i=1,2$
to be parallel:
$$K_1:L_1:M_1:N_1=K_2:L_2:M_2:N_2$$
or equivalently
$$\frac{K_1}{K_2}=\frac{L_1}{L_2}=\frac{M_1}{M_2}=\frac{N_1}{N_2}.$$

\formula 10:
Condition for two lines
$(x-x_i)/ K_i=(y-y_i)/ L_i=(z-z_i)/M_i=(w-w_i)/N_i$, $i=1,2$
to be perpendicular
(\cite[p.~68]{Frost}):
$$a_1^2(L_1M_2+L_2M_1)+a_2^2(K_1M_2+K_2M_1)+a_3^2(K_1L_2+K_2L_1)+$$
$$b_1^2(K_1N_2+K_2N_1)+b_2^2(L_1N_2+L_2N_1)+b_3^2(M_1N_2+M_2N_1)=0.$$

\formula 11:
Condition for two lines
$(x-x_i)/ K_i=(y-y_i)/ L_i=(z-z_i)/M_i=(w-w_i)/N_i$, $i=1,2$
to intersect (or be parallel):
$$
\begin{vmatrix}
x_1&y_1&z_1&w_1\\
K_1&L_1&M_1&N_1\\
x_2&y_2&z_2&w_2\\
K_2&L_2&M_2&N_2\\
\end{vmatrix}
=0
$$
or equivalently
$$
\begin{vmatrix}
x_2-x_1&y_2-y_1&z_2-z_1\\
K_1&L_1&M_1\\
K_2&L_2&M_2\\
\end{vmatrix}
=0.
$$

\formula 12: Direction cosines of the line
$(x-x_0)/ K=(y-y_0)/ L=(z-z_0)/M=(w-w_0)/N$:
$$\frac{KF_1}{3V\sigma},\ \frac{LF_2}{3V\sigma},
\ \frac{MF_3}{3V\sigma},\ \frac{NF_4}{3V\sigma},$$
where $V$ is the volume of the reference tetrahedron,
$F_i$ is the area of face $i$ (the face opposite vertex $A_i$),
and $\sigma$ is determined from
$$a_1^2LM+a_2^2MK+a_3^2KL+b_1^2KN+b_2^2LN+b_3^2MN=-\sigma^2.$$
The direction cosines are the cosines of the angles that the line
makes with the normals to the four faces of the reference tetrahedron.
They are proportional to the direction vector.

\goodbreak
\textbf{PLANES}

\formula 13:
General equation of a plane
(\cite[p.~69]{Frost}):
$$Ax+By+Cz+Dw=0$$
where not all coefficients are 0.
The coefficients $A$, $B$, $C$, $D$, are proportional to the directed distances
from the plane to the vertices of the reference tetrahedron.

\formula 14:
Equation of the plane through 3 points, $(x_i,y_i,z_i,w_i)$, $i=1,2,3$
(\cite{Frost}):
$$
\begin{vmatrix}
x&y&z&w\\
x_1&y_1&z_1&w_1\\
x_2&y_2&z_2&w_2\\
x_3&y_3&z_3&w_3\\
\end{vmatrix}
=0.$$

\formula 15:
Condition for planes $A_1x+B_1y+C_1z+D_1w=0$ and $A_2x+B_2y+C_2z+D_2w=0$
to be parallel
(\cite[p.~70]{Frost}):
$$\frac{A_1-D_1}{A_2-D_2}=\frac{B_1-D_1}{B_2-D_2}=\frac{C_1-D_1}{C_2-D_2}.$$

\formula 16:
Equation of the plane through the point $(x_1,y_1,z_1,w_1)$
and the line $(x-x_2)/K=(y-y_2)/L=(z-z_2)/M=(w-w_2)/N$:
$$
\begin{vmatrix}
x&y&z&w\\
x_1&y_1&z_1&w_1\\
x_2&y_2&z_2&w_2\\
K&L&M&N\\
\end{vmatrix}
=0.
$$

\formula 17:
Equation of plane through the line 
$(x-x_1)/K_1=(y-y_1)/L_1=(z-z_1)/M_1=(w-w_1)/N_1$
and parallel to the line with
direction vector $(K_2,L_2,M_2,N_2)$:
$$
\begin{vmatrix}
x&y&z&w\\
x_1&y_1&z_1&w_1\\
K_1&L_1&M_1&N_1\\
K_2&L_2&M_2&N_2\\
\end{vmatrix}
=0.
$$

\formula 18:
Point of intersection of the line
$(x-x_1)/K=(y-y_1)/L=(z-z_1)/M=(w-w_1)/N$
and the plane $Ax+By+Cz+Dw=0$:
$$(x_1-rK,y_1-rL,z_1-rM,w_1-rN)$$
where
$$r=\frac{Ax_1+By_1+Cz_1+Dw_1}{AK+BL+CM+DN}.$$

\formula 19:
Condition for the line
$(x-x_1)/K=(y-y_1)/L=(z-z_1)/M=(w-w_1)/N$
to be parallel to the plane $Ax+By+Cz+Dw=0$:
$$AK+BL+CM+DN=0.$$

\section{Tetrahedron Centers}
\label{section:tetrahedronCenters}

In this appendix, we collect together information about various ``centers'' associated
with a tetrahedron.
We give the tetrahedral coordinates for the more well-known
such centers and explain why some of these centers were not included in our study.

\goodbreak
\textbf{CENTROID}

The centroid, $G$, of a tetrahedron (\cite[p.~54]{Altshiller}) is the center of gravity of unit
masses placed on the vertices. Thus it has barycentric coordinates of
$$G=(1,1,1,1).$$
The exact tetrahedral coordinates are $(1/4,1/4,1/4,1/4)$.
The centroid is also the intersection of the medians of the
tetrahedron (the lines from a vertex to the centroid of the opposite face).

\goodbreak
\textbf{INCENTER}

The incenter, $I$, of a tetrahedron (\cite[p.~76]{Altshiller}) is the center
of the sphere
inscribed in the tetrahedron (touching each of the faces internally).
If we let $r$ be the inradius of the tetrahedron, then the volume of
$IA_2A_3A_4$ is $\frac{1}{3}rF_1$. Similarly for the other three volumes
formed by $I$ and the faces of the tetrahedron.  These four volumes sum
to $V$, the volume of the tetrahedron. Thus
$r=3V/(F_1+F_2+F_3+F_4)$.
The incenter is equidistant from each face of the tetrahedron.  Thus, the tetrahedral
coordinates are
$$I=(F_1,F_2,F_3,F_4).$$
To convert to exact tetrahedral coordinates, each coordinate should
be divided by $F$, the surface area of the tetrahedron.

\goodbreak
\textbf{CIRCUMCENTER}

The circumcenter, $O$, of a tetrahedron (\cite[p.~56]{Altshiller}) is the center of the
circumscribed sphere. If $(O_x,O_y,O_z,O_w)$ are the coordinates
for the circumcenter of our reference tetrahedron, and if $R$
denotes the circumradius, then we 
can set up 4 equations in $O_x$, $O_y$, $O_z$, $O_w$, and $R$:
$$
\begin{aligned}
d(O,A_1)^2&=R^2\\
d(O,A_2)^2&=R^2\\
d(O,A_3)^2&=R^2\\
d(O,A_4)^2&=R^2\\
\end{aligned}
$$
where $d(P_1,P_2)$ denotes the distance between points $P_1$ and $P_2$.
This distance formula is given by formula 4.
Upon subtracting equation 2 from equation 1, equation 3 from equation 2,
and equation 4 from equation 3, we get 3 linear equations
in $O_x$, $O_y$, $O_z$, and $O_w$. Combining these with
$O_x+O_y+O_z+O_w=1$
gives us 4 linear equations in 4 unknowns. These are straightforward
to solve.
The value of $O_x$ found is proportional to
$$O_x=a_1^2b_1^2(b_2^2+b_3^2-a_1^2)+
a_2^2b_2^2(b_3^2+a_1^2-b_2^2)+
a_3^2b_3^2(a_1^2+b_2^2-b_3^2)
-2a_1^2b_2^2b_3^2.$$
The values of $O_y$, $O_z$, and $O_w$ are similar and can be obtained
from $O_x$ by the mappings given in display (5).
Specifically,
$$
\begin{aligned}
O_y&=a_1^2b_1^2(a_2^2+b_3^2-b_1^2)+
a_2^2b_2^2(b_3^2+b_1^2-a_2^2)+
a_3^2b_3^2(b_1^2+a_2^2-b_3^2)
-2b_1^2a_2^2b_3^2,\\
O_z&=a_1^2b_1^2(b_2^2+a_3^2-b_1^2)+
a_2^2b_2^2(a_3^2+b_1^2-b_2^2)+
a_3^2b_3^2(b_1^2+b_2^2-a_3^2)
-2b_1^2b_2^2a_3^2,\\
O_w&=a_1^2b_1^2(a_2^2+a_3^2-a_1^2)+
a_2^2b_2^2(a_3^2+a_1^2-a_2^2)+
a_3^2b_3^2(a_1^2+a_2^2-a_3^2)
-2a_1^2a_2^2a_3^2.\\
\end{aligned}
$$

\goodbreak
\textbf{MONGE POINT}

The Monge point, $M$, of a tetrahedron (\cite[p.~76]{Altshiller})
is the common intersection point of the six planes through
the midpoints of the edges of the tetrahedron and perpendicular
to the opposite edges. The Monge point is the symmetric of the
circumcenter with respect to the centroid (\cite[p.~77]{Altshiller})
and thus its coordinates can be found from them:
$$M=2G-O.$$

\begin{definition}
The three points $G$, $O$, and $M$ lie on a straight line
called the {\it Euler line} of the tetrahedron (\cite[p.~77]{Altshiller}).
\end{definition}

\goodbreak
\textbf{EULER POINT}

The Euler point, $E$, corresponds to the nine-point center in the plane.
It is frequently called the twelve point center (\cite[p.~289]{Altshiller})
because it is the center of a sphere that passes through 12 notable
points in the tetrahedron.
The Euler point lies on the Euler line of the tetrahedron and
divides the segment $MO$ in the ratio $1:2$ and it divides
the segment $GM$ in the ratio $1:2$. Thus its coordinates
can be found from the coordinates of those points:
$$E=(2G+M)/3.$$

\goodbreak
\textbf{ORTHOCENTER}

The altitudes of a tetrahedron do not normally intersect.
They intersect if and only if the tetrahedron is orthocentric and
in that case, the intersection point (the orthocenter)
coincides with the Monge point of the tetrahedron (\cite[p.~71]{Altshiller}).
We do not include the orthocenter of a tetrahedron
as a distinguished point in our study since it is not present
in all tetrahedra.

\goodbreak
\textbf{SYMMEDIAN POINT (Lemoine Point)}

In a tetrahedron, the cevians to the symmedian points on the opposite faces
do not normally intersect.
They intersect if and only if the tetrahedron is isodynamic
(\cite[p.~315]{Altshiller}).
The point of intersection is called the symmedian point of the tetrahedron.
We do not include the symmedian point of a tetrahedron
as a distinguished point in our study since it is not present
in all tetrahedra.

\goodbreak
\textbf{GERGONNE POINT}

In a tetrahedron, the cevians to the Gergonne points on the opposite faces
do not normally intersect.
They intersect if and only if the tetrahedron is circumscriptible
(\cite[p.~299]{Altshiller}).
The point of intersection is called the Gergonne point of the tetrahedron.
We do not include the Gergonne point of a tetrahedron
as a distinguished point in our study since it is not present
in all tetrahedra.

\goodbreak
\textbf{NAGEL POINT}

In a tetrahedron, the cevians to the Nagel points on the opposite faces
do not normally intersect.
They intersect if and only if the tetrahedron is circumscriptible
(\cite[p.~299]{Altshiller}).
The point of intersection is called the Nagel point of the tetrahedron.
We do not include the Nagel point of a tetrahedron
as a distinguished point in our study since it is not present
in all tetrahedra.

\goodbreak
\textbf{FERMAT POINT}

In a tetrahedron, the cevians to the points of tangency of the opposite faces
with the insphere do not normally intersect.
If they intersect, the tetrahedron is called an isogonic tetrahedron
(\cite[p.~328]{Altshiller}).
The point of intersection is called the Fermat point of the tetrahedron.
In this case, the insphere touches the faces at the Fermat point
of each face.
We do not include the Fermat point of a tetrahedron
as a distinguished point in our study since it is not present
in all tetrahedra.

\goodbreak

\end{document}